\documentclass[12pt, reqno]{amsart}
\usepackage{amsmath, amsthm, amscd, amsfonts, amssymb, graphicx, color}
\usepackage[bookmarksnumbered, colorlinks, plainpages]{hyperref}
\hypersetup{colorlinks=true,linkcolor=red, anchorcolor=green, citecolor=cyan, urlcolor=red, filecolor=magenta, pdftoolbar=true}
\usepackage{tikz}
\usetikzlibrary{arrows,decorations.pathmorphing,backgrounds,positioning,fit}

\newtheorem{theorem}{Theorem}[section]
\newtheorem{lemma}[theorem]{Lemma}
\newtheorem{proposition}[theorem]{Proposition}
\newtheorem{corollary}[theorem]{Corollary}
\theoremstyle{definition}

\theoremstyle{remark}
\newtheorem{remark}[theorem]{Remark}
\numberwithin{equation}{section}

\begin{document}

\setcounter{page}{1}

\title[Phillips symmetric operators and their extensions
]{Phillips symmetric operators and their extensions}

\author[S. Kuzhel \MakeLowercase{and} L. Nizhnik]{Sergii Kuzhel,$^1$$^{*}$ \MakeLowercase{and} Leonid Nizhnik$^2$}

\address{$^{1}$AGH University of Science and Technology, Krak\'{o}w 30-059, Poland.}
\email{\textcolor[rgb]{0.00,0.00,0.84}{kuzhel@agh.edu.pl}}

\address{$^{2}$Institute of Mathematics  National Academy of Sciences of Ukraine, Kyiv  01-601,  Ukraine}
\email{\textcolor[rgb]{0.00,0.00,0.84}{nizhnik@imath.kiev.ua}}

\subjclass[2010]{Primary 47B25; Secondary 47A10}

\keywords{symmetric operator, characteristic function, wandering subspace, bilateral shift, Lebesgue spectrum.}

\begin{abstract}
Let $S$ be a symmetric operator with equal defect numbers and let $\mathfrak{U}$
be a set of unitary operators in a Hilbert space $\mathfrak{H}$.
The operator $S$ is called $\mathfrak{U}$-invariant if
$US=SU$ for all $U\in\mathfrak{U}$. Phillips \cite{PH} constructed an
example of $\mathfrak{U}$-invariant symmetric operator $S$  which has
no $\mathfrak{U}$-invariant self-adjoint extensions. It was discovered
that such symmetric operator has a constant characteristic function \cite{KO}.
For this reason, each symmetric operator $S$  with constant characteristic function is called
a \emph{Phillips symmetric operator}.

The paper is devoted to the investigation of self-adjoint (and, more generally, proper) extensions of a Phillips symmetric operator.
Such extensions differ from those that are commonly studied in the literature and they have a lot of curious properties.
In particular, proper extensions of a Phillips symmetric operator that have real spectrum are similar to each other.
\end{abstract} 
\maketitle

\section{Introduction}
Let $S$ be a symmetric operator with equal defect numbers and let $\mathfrak{U}$
be a family of unitary operators in a Hilbert space $\mathfrak{H}$ such that the inclusion $U\in\mathfrak{U}$ implies
$U^*\in\mathfrak{U}$. The operator $S$ is called $\mathfrak{U}$-invariant if
$S$ commutes with all $U\in\mathfrak{U}$. Does there exist at least one $\mathfrak{U}$-invariant self-adjoint extension of
$S$?  The answer is definitely affirmative if $S$  is assumed to be semibounded and the Friedrichs extension of $S$ gives the required example.

In general case of non-semibounded operators, R. Phillips constructed a symmetric operator $S$
 and a family $\mathfrak{U}$ of unitary operators commuting with $S$ such that the
$\mathfrak{U}$-invariant $S$ has no $\mathfrak{U}$-invariant self-adjoint extensions \cite[p. 382]{PH}.
Precisely, the mentioned symmetric operator $S$ acts in the Hilbert space
$l_2(\mathbb{Z})$ and it is defined as the Cayley transform
$S=i(V+I)(V-1)^{-1}$  of the isometric right shift operator
 $$
 V\{x_n\}_{n\in\mathbb{Z}}=\{x_{n+1}\}_{n\in\mathbb{Z}}, \quad  \mathcal{D}(V)=\{\{x_n\}_{n\in\mathbb{Z}}\in{l}_2(\mathbb{Z}) \ : \ x_0=0\}.
 $$
 The family $\mathfrak{U}$ consists of unitary operators $U_\theta\{x_n\}_{n\in\mathbb{Z}}=\{y_n\}_{n\in\mathbb{Z}}$ ($|\theta|=1$), where $y_n=\theta{x_n}$  $(n\in\mathbb{N})$ and $y_n=x_n$  ($n\in\mathbb{Z}\setminus\mathbb{N}$).
 The operator $S$ is $\mathfrak{U}$-invariant but there are no $\mathfrak{U}$-invariant self-adjoint extensions of $S$.

It was discovered \cite{KO} that the characteristic function of the symmetric operator
constructed in the Phillips work is a constant in the upper half-plane $\mathbb{C}_+$. This fact
can be used for the general definition of Phillips symmetric operators.  Namely, we  will say that a
symmetric operator $S$ with equal defect numbers is \emph{a Phillips symmetric operator} (PSO)
if its characteristic function is an operator-constant on $\mathbb{C}_+$.

The concept of characteristic function of a symmetric operator was firstly introduced by Shtraus
\cite{SH} and,  further,  substantially developed by Kochubei \cite{Kochubei}  on the base of boundary triplets  technique  \cite{Gor}.
Section 2 contains all necessary results about characteristic functions  which are used in the paper.

The present paper is devoted to the investigation of PSO as well as
theirs self-adjoint (and, more generally, proper\footnote{an extension $A$ of a symmetric operator $S$ is called \emph{proper}  if $S\subset{A}\subset{S^*}$}
extensions).
Such self-adjoint extensions differ from those that are commonly studied in the literature
\cite{AL} and they have a lot of curious properties.

Our original definition of PSO deals with the concept of characteristic function.
In many cases, an explicit calculation of a characteristic function is technically complicated.
For this reason, in Section 3, we establish equivalent descriptions of PSO 
 (Theorems \ref{assa89}, \ref{assa77}, Corollary \ref{bbb111})  which can be employed as independent definitions of PSO.
These results lead to the conclusion that each simple\footnote{a symmetric operator is called \emph{simple}
if its restriction to any nontrivial reducing subspace is not a self-adjoint operator} PSO coincides with
orthogonal sum of simple maximal symmetric operators having the same nonzero defect numbers in
upper $\mathbb{C}_+$ and lower $\mathbb{C}_-$ half planes.  Such kind
of decomposition means that every simple PSO  $S$ is unitary equivalent to the momentum operator
with one point interaction
$$
\textsc{S}=i\frac{d}{dx}, \qquad \mathcal{D}(\textsc{S})=\{u\in{W^1_2(\mathbb{R}, N)} : u(0)=0\}
$$
acting in the Hilbert space $L_2(\mathbb{R}, N)$,  where the dimension of the auxiliary Hilbert space  $N$
coincides with the defect number of $S$.

Section 4 is devoted to proper extensions of PSO. The main result (Theorem \ref{assa102})  states
that all proper extensions of a PSO $S$ with real spectrum are similar to each other.  In fact we can say more:  each
proper extension with real spectrum can be interpreted as self-adjoint extension of $S$  for a special choice of inner
product equivalent to the initial one.

Some properties of PSO with defect numbers $<1, 1>$  were established  in \cite{Arlin}.
In particular, analogues of Theorems \ref{assa77}, \ref{assa102},  and Corollary \ref{assa17}  were proved. 
  
In Section 5, PSO are determined as the restrictions of a given self-adjoint operator
$A$. According to Theorem \ref{assa108},
Phillips symmetric operators which can be obtained in this way are in one-to-one
correspondence with the wandering subspaces $\mathcal{L}$ of the Cayley transform $U$ of $A$.
This means that the set of restrictions of $A$ contains PSO   only in the case where
$A$  has a  reducing subspace $\mathfrak{H}_0$ such that $A_0=A\upharpoonright_{\mathcal{D}(A)\cap\mathfrak{H}_0}$
is a self-adjoint operator in $\mathfrak{H}_0$ with Lebesgue spectrum on $\mathbb{R}$.
The existence of a simple PSO  is equivalent to the  fact  that
$A$  has Lebesgue spectrum on $\mathbb{R}$ (Corollary \ref{bbb27}).

 In Section 6, examples of PSO are considered. We establish a useful (in our opinion)
 characterization of wavelets as functions from the defect subspace $\mathfrak{N}_{-i}$
of the specially chosen PSO (Proposition \ref{neww35}).

Results of Sections 4-6 show that one-point interaction of the momentum operator:
 $i\frac{d}{dx}+\alpha\delta(x-y)$
leads to self-adjoint operators with Lebesgue spectrum which are unitary equivalent to each other.
This means that one should consider more complicated  perturbations of the momentum operator
for the construction of self-adjoint operators with non-trivial spectral properties.
 In this way, self-adjoint  momentum operators
acting in two intervals were studied in \cite{pedersen1, pedersen2}.  The momentum operators
defined on oriented metric graphs were investigated in \cite{exner}.
General nonlocal point interactions for first order differential operators
were introduced and  studied in \cite{Nizhnik, Nizhnik2}.

In Section 7,  we continue investigations of \cite{Nizhnik}
and to focus on special classes of perturbations which can be  characterized as
one point interaction defined by the nonlocal potential $\gamma\in{L_2(\mathbb{R})}$.

\section{Characteristic functions of symmetric operators}\label{chap2}
Let $A$ be a linear operator acting in a
Hilbert space ${\mathfrak H}$. Its domain is denoted $\mathcal{D}(A)$,
while $A\upharpoonright_{\mathcal{D}}$ stands for the restriction of
$A$ onto a set $\mathcal{D}$.
Here and in the following we denote by $\mathbb C_+$ ($\mathbb C_-$)
the open upper (resp.\ lower) half plane.

{\bf I.}  Let $S$ be a closed symmetric densely defined operator with equal defect numbers acting in a separable Hilbert space
$\mathfrak{H}$ with inner product $(\cdot,\cdot)$ linear in the first argument.

We denote by $\mathfrak{N}_\lambda=\ker(S^*-\lambda{I})$
the defect subspaces of $S$  and consider the linear spaces
$$
\mathfrak{M}_\lambda=\mathfrak{N}_{\lambda}\dot{+}\mathfrak{N}_{\overline{\lambda}},   \qquad  \lambda\in\mathbb{C}\setminus\mathbb{R}.
$$

According to the von Neumann formulas (see, e.g., \cite{KK, Schm}) each proper extension
$A$ of $S$  is
uniquely determined by the choice of a subspace
$M\subset\mathfrak{M}_\lambda$:
\begin{equation}\label{e55}
A=S^*\upharpoonright_{\mathcal{D}(A)}, \qquad \mathcal{D}(A)=\mathcal{D}(S)\dot{+}M,
\end{equation}

Let us set $M=\mathfrak{N}_\lambda$  in \eqref{e55} and denote by
\begin{equation}\label{as908}
A_\lambda=S^*\upharpoonright_{\mathcal{D}(A_\lambda)},  \quad \mathcal{D}(A_\lambda)=\mathcal{D}(S)\dot{+}\mathfrak{N}_\lambda,   \quad \lambda\in\mathbb{C}\setminus\mathbb{R}
\end{equation}
the corresponding proper extensions of $S$.
The operators $\textsf{sign}(Im \ \lambda)A_\lambda$ are maximal
dissipative\footnote{An operator $A$ is called dissipative if $Im (Af,f)\geq{0}$ for all $f\in\mathcal{D}(A)$ and
maximal dissipative if there are no dissipative extensions of $A$.} and $A_\lambda^*=A_{\overline{\lambda}}$.
The resolvent set of every maximal dissipative operator contains $\mathbb{C}_-$. For this reason, the operator-function
\begin{equation}\label{as909}
\textsf{Sh}(\lambda)=(A_\lambda-iI)(A_\lambda+iI)^{-1}\upharpoonright_{\mathfrak{N}_i} : {\mathfrak{N}_i}\to{\mathfrak{N}_{-i}}, \quad \lambda\in\mathbb{C}_+
\end{equation}
is well-defined and it coincides with the characteristic function of symmetric operator $S$ defined by A. Shtraus \cite{SH}.

Another (equivalent) definition of $\textsf{Sh}(\cdot)$ in \cite{SH} is based on the relation
\begin{equation}\label{GDR2}
\mathcal{D}(A_\lambda)=\mathcal{D}(S)\dot{+}\mathfrak{N}_\lambda=\mathcal{D}(S)\dot{+}(I-\textsf{Sh}(\lambda))\mathfrak{N}_i, \qquad \lambda\in\mathbb{C}_+,
\end{equation}
which allows one to determine uniquely $\textsf{Sh}(\cdot)$.

The explicit construction of $\textsf{Sh}(\cdot)$ deals with the calculation of $\mathfrak{N}_\lambda$ that, sometimes, is technically complicated.
This inconvenience was overcame in \cite{Kochubei} with the use of boundary triplet
technique.  We recall \cite{KK, Oliveira} that a triplet $(\mathcal{H}, \Gamma_-, \Gamma_+)$, where $\mathcal{H}$ is an
auxiliary Hilbert space and $\Gamma_\pm$ are linear
mappings of $\mathcal{D}(S^*)$ into $\mathcal{H}$, is called a
\emph{boundary triplet of} $S^*$ if
\begin{equation}\label{assa5}
(S^*f, g)-(f, S^*g)=i[(\Gamma_{+}f,\Gamma_{+}g)_{\mathcal
H}-(\Gamma_{-}f,\Gamma_{-}g)_{\mathcal H}],  \quad  f, g\in\mathcal{D}(S^*)
\end{equation}
holds and the map $(\Gamma_-,
\Gamma_+):\mathcal{D}(S^*)\to\mathcal{H}\oplus\mathcal{H}$ is
surjective.

Let a boundary triplet $(\mathcal{H}, \Gamma_-, \Gamma_+)$ be given.
Then the domains of definition of operators $A_\lambda$  in \eqref{as908} admit the presentation
\begin{equation}\label{new1}
\mathcal{D}(A_\lambda)=\left\{f\in\mathcal{D}(S^*) : \begin{array}{cc}
\Theta(\lambda)\Gamma_+f=\Gamma_-f,  &  \lambda\in{\mathbb{C}_+} \\
\Gamma_+f=\Theta(\lambda)\Gamma_-f,  &  \lambda\in{\mathbb{C}_-}
\end{array}\right\}
\end{equation}
where $\Theta(\cdot)$ is an operator in $\mathcal{H}$.

The operator-valued function $\Theta(\cdot)$ defined on $\mathbb{C}\setminus\mathbb{R}$
is called \emph{the characteristic function of $S$ associated
with boundary triplet $(\mathcal{H}, \Gamma_-, \Gamma_+)$}.
It follows from the relation $A_\lambda^*=A_{\overline{\lambda}}$ and \eqref{assa5}
that $\Theta^*(\lambda)=\Theta(\overline{\lambda})$.

The explicit form of characteristic function depends on the choice of a boundary triplet.
However, in any case, $\Theta(\cdot)$ is a holomorphic operator-valued function whose values
are strong contractions in $\mathcal{H}$ (i.e., $\|\Theta(\lambda)\|<1$) \cite{Kochubei}.

The characteristic function determines a simple symmetric operator up to unitary equivalence. Namely, the following result
holds:
\begin{theorem}[\cite{Kochubei}]\label{assa44}
Simple symmetric operators $S_1$ and $S_2$ are unitary equivalent if and only if some of their
characteristic functions coincide.
\end{theorem}

The Shtraus characteristic function $\textsf{Sh}(\cdot)$ defined in \eqref{as909} coincides (up to the multiplication by unitary operator)
with $\Theta(\cdot)$ for special choice of
boundary triplet.  Precisely, the simplest (inspired by the von Neumann formulas)
boundary triplets $(\mathfrak{N}_{\mu}, \Gamma_-, \Gamma_+)$
 of $S^*$  can be constructed as follows
\begin{equation}\label{sas6}
\Gamma_-f=\sqrt{2{Im \ \mu}}Vf_{\overline{\mu}}, \quad \Gamma_+{f}=\sqrt{2{Im \ \mu}}f_{\mu}, \
f=u+f_{\mu}+f_{\overline{\mu}}\in\mathcal{D}(S^*),
\end{equation}
where $\mu\in\mathbb{C}_+$  and  $V : \mathfrak{N}_{\overline{\mu}}\to\mathfrak{N}_{\mu}$ is an arbitrary unitary
mapping. Assume that $\mu=i$. Then, the characteristic function $\Theta(\cdot)$ associated with
the boundary triplet $(\mathfrak{N}_{i}, \Gamma_-, \Gamma_+)$  coincides with the function
$-V\textsf{Sh}(\cdot)$ on $\mathbb{C}_+$.

\begin{remark}\label{rem22}
There are various  approaches to the definition of boundary triplets. For instance,
\cite{Gor, Schm},  a triplet $(\mathcal{H}, \Gamma_0, \Gamma_1)$  where
$\Gamma_0$, $\Gamma_1$  are linear
mappings of $\mathcal{D}(S^*)$ into $\mathcal{H}$,  is called a
\emph{boundary triplet of} $S^*$ if the Green's identity
$$
(S^*f, g)-(f, S^*g)=(\Gamma_1f, \Gamma_0g)_{\mathcal{H}}-(\Gamma_0f,
\Gamma_1g)_{\mathcal{H}}, \quad  f, g\in\mathcal{D}(S^*)
$$
holds and the map $(\Gamma_0,
\Gamma_1):\mathcal{D}(S^*)\to\mathcal{H}\oplus\mathcal{H}$ is
surjective.

The operators $\Gamma_\pm$  in \eqref{assa5} and $\Gamma_i$ are related as follows
$\Gamma_{\pm}=\frac{1}{\sqrt{2}}(\Gamma_1{\pm}i\Gamma_0)$
and, obviously, the definitions of boundary triplets $(\mathcal{H}, \Gamma_-, \Gamma_+)$
and $(\mathcal{H}, \Gamma_0, \Gamma_1)$ are equivalent.
\end{remark}

{\bf II.} The characteristic function $\Theta(\cdot)$  admits a natural interpretation in the Krein space setting
(see \cite{AK_AK, AZ} for the basic theory of Krein spaces and terminology).
To explain this point,  we fix a boundary triplet $(\mathcal{H}, \Gamma_-, \Gamma_+)$
and rewrite \eqref{assa5} as:
\begin{equation}\label{new24}
(S^*f,g)-(f,S^*g)=i[\Psi{f},\Psi{g}],
\end{equation}
where
\begin{equation}\label{new111}
\Psi=\begin{bmatrix}
\Gamma_+ \\
\Gamma_-
\end{bmatrix}  : \mathcal{D}(S^*) \to \textsf{H}=\begin{bmatrix}
\mathcal{H} \\
\mathcal{H}
\end{bmatrix},
\end{equation}
maps $\mathcal{D}(S^*)$ into the Krein space $(\textsf{H},
[\cdot,\cdot])$ with the indefinite inner product
\begin{equation}\label{new30}
[\textsf{x},\textsf{y}]=(x_0, y_0)-(x_1, y_1), \qquad \textsf{x}=\begin{bmatrix}
x_0 \\
x_1
\end{bmatrix}, \ \textsf{y}=\begin{bmatrix}
y_0 \\
y_1
\end{bmatrix}\in\textsf{H}.
\end{equation}

It follows from the definition of boundary triplets that the mapping
$\Psi : \mathcal{D}(S^*) \to \textsf{H}$ is surjective and $\ker\Psi=\mathcal{D}(S)$.

In view of \eqref{new30}, the fundamental decomposition of the Krein space $(\textsf{H},
[\cdot,\cdot])$ coincides with
\begin{equation}\label{assa3}
\textsf{H}=\textsf{H}_+\oplus\textsf{H}_-, \qquad \textsf{H}_+=\begin{bmatrix}
\mathcal{H} \\
0
\end{bmatrix}, \quad  \textsf{H}_-=\begin{bmatrix}
0 \\
\mathcal{H}
\end{bmatrix},
\end{equation}
where $\textsf{H}_+=\Psi\ker\Gamma_-$ and $\textsf{H}_-=\Psi\ker\Gamma_+$  are respectively, maximal uniformly positive and negative subspaces
with respect to the indefinite inner product $[\cdot,\cdot]$.

By virtue of \eqref{e55}, each proper extension $A$ of $S$
is completely determined by a subspace
$\textsf{L}=\Psi\mathcal{D}(A)=\Psi(\mathcal{D}(S)\dot{+}M)=\Psi{M}$ of $\textsf{H}$.
In other words, there is a one-to-one correspondence between
subspaces of $\textsf{H}$ and proper extensions of $S$. In particular,
proper extensions $A_\lambda$ in \eqref{as908} are determined by the subspaces
$\textsf{L}_\lambda=\Psi\mathcal{D}(A_\lambda)$, which are maximal uniformly positive ($\lambda\in\mathbb{C}_+$) and
maximal uniformly negative ($\lambda\in\mathbb{C}_-$)  in $(\textsf{H}, [\cdot,\cdot])$, see \cite{HK}.

Taking \eqref{new1} and \eqref{new111} into account, we arrive at the conclusion that
the maximal uniformly positive subspace $\textsf{L}_\lambda$  is decomposed with respect to the fundamental
decomposition (\ref{assa3}):
$$
\textsf{L}_\lambda=\Psi\mathcal{D}(A_\lambda)=\left\{\begin{bmatrix}
\Gamma_+f \\
\Theta\Gamma_+f
\end{bmatrix}  : f\in\mathcal{D}(A_\lambda)\right\}=\{\textsf{h}_++\widetilde{\Theta}(\lambda)\textsf{h}_+ : \textsf{h}_+\in\textsf{H}_+\},
$$
where $\widetilde{\Theta}(\cdot) : \textsf{H}_+ \to \textsf{H}_- $  acts as follows:
\begin{equation}\label{ee2}
\widetilde{\Theta}(\lambda)\textsf{h}_+=\widetilde{\Theta}(\lambda)\begin{bmatrix}
h \\
0
\end{bmatrix}=\begin{bmatrix}
0 \\
\Theta(\lambda)h
\end{bmatrix}, \qquad \lambda\in\mathbb{C}_+.
\end{equation}
This means that
$\widetilde{\Theta}(\lambda)$ is the angular operator of the maximal uniformly positive subspace $\textsf{L}_\lambda$ with respect to the maximal uniformly positive subspace
$\textsf{H}_+$ of the fundamental decomposition \eqref{assa3}  (see \cite{AZ} for the concept of angular operators).

Self-adjoint extensions $A$ of $S$ correspond to hypermaximal neutral subspaces $\textsf{L}=\Psi\mathcal{D}(A)$ of the Krein
space $(\textsf{H}, [\cdot,\cdot])$. Each hypermaximal neutral subspace is determined uniquely by a unitary mapping between
subspaces  $\textsf{H}_+$ and $\textsf{H}_-$ of the fundamental decomposition \eqref{assa3}. This fact leads to
the conclusion that each self-adjoint extension $A$ of $S$ can be described as
$$
{A}=S^*\upharpoonright_{\mathcal{D}({A})},
\qquad \mathcal{D}({A}) =\{f\in\mathcal{D}(S^*) \ : \
\mathbf{T}\Gamma_+f=\Gamma_-f \},
$$
where $\mathbf{T}$ is a unitary operator in $\mathcal{H}$.

{\bf III.} The explicit form of characteristic function depends on the choice of boundary triplet.
Let $\Theta_i(\cdot)$  $(i=1,2)$  be characteristic functions
associated with boundary triplets $({\mathcal{H}_i, \Gamma_-^i, \Gamma_+^i})$.
Since the dimensions of the auxiliary Hilbert spaces $\mathcal{H}_i$
coincide with the defect number of $S$, without loss of generality,
we may assume that $\mathcal{H}_1=\mathcal{H}_2=\mathcal{H}$.

 It is easy to see that  the operator $K : \textsf{H} \to \textsf{H}$  defined by the formula
$$
K\begin{bmatrix}
 \Gamma_{+}^1f \\
 \Gamma_{-}^1f
\end{bmatrix}=\begin{bmatrix}
 \Gamma_{+}^2f \\
 \Gamma_{-}^2f
\end{bmatrix},   \quad f\in\mathcal{D}(S^*).
$$
is surjective in $\textsf{H}$ and, moreover, $K$  is a unitary operator in the Krein space
 $(\textsf{H}, [\cdot,\cdot])$,  i.e.
$[K\textsf{x},K\textsf{y}]=[\textsf{x},\textsf{y}]$, $\textsf{x},\textsf{y}\in\textsf{H}$  (the latter relation follows from \eqref{new24} - \eqref{new30}).
Each unitary operator $K$ in $(\textsf{H}, [\cdot,\cdot])$ determines the so-called interspherical linear fractional transformation \cite[Chapter III, section 3]{AZ}
$$
\Phi_K(Z)=(K_{21}+K_{22}Z)(K_{11}+K_{12}Z)^{-1}, \qquad K=\begin{bmatrix}
K_{11} & K_{12} \\
K_{21} &  K_{22}
\end{bmatrix},
$$
where $K_{ij}$ are operator components of decomposition of $K$  with respect to \eqref{assa3}
and a bounded linear operator $Z$ maps $\textsf{H}_+$  into $\textsf{H}_-$.
The  interspherical  transformation $\Phi_K(Z)$ is well defined for all  $Z : \textsf{H}_+\to\textsf{H}_-$
with $\|Z\|\leq{1}$  (i.e.,  $0\in\rho(K_{11}+K_{12}Z)$ )  and $\|\Phi_K(Z)\|\leq{1}$.

It is known  \cite{Kochubei, KK} that
\begin{equation}\label{assa4}
\widetilde{\Theta}_2(\lambda)=\Phi_K(\widetilde{\Theta}_1(\lambda)), \qquad \lambda\in\mathbb{C}_+,
\end{equation}
where $\widetilde{\Theta}_i(\cdot) : \textsf{H}_+\to\textsf{H}_- $ are defined similarly to
\eqref{ee2}.

\section{Phillips symmetric operator}

We say that a symmetric operator with equal nonzero defect numbers is a \emph{Phillips symmetric operator \ (PSO)} if its characteristic function
$\Theta(\cdot)$ is an operator-constant on $\mathbb{C}_+$.

By virtue of \eqref{assa4},  this definition  does not depend on the choice of boundary triplet.
However,  in many cases,  it is not easy to use it  (because one should to calculate the characteristic function).
For this reason a series of statements which can be used as
(equivalent) definitions of PSO are presented below.

\begin{theorem}\label{assa89}
A symmetric operator $S$ with equal defect numbers is a Phillips symmetric operator if and only if
\begin{equation}\label{assa8}
\mathfrak{N}_\lambda\subset\mathcal{D}(S)\dot{+}\mathfrak{N}_\mu,  \quad  \mbox{for all} \quad  \lambda, \mu\in\mathbb{C}_+.
\end{equation}
\end{theorem}
\begin{proof}
If $S$ is a PSO, then its
characteristic function $\Theta(\cdot)$ associated with the boundary triplet $(\mathfrak{N}_i, \Gamma_-, \Gamma_+)$
determined by \eqref{sas6} has to be a constant. Therefore,  the Shtraus characteristic function
$\textsf{Sh}(\lambda)$ coincides with an operator $U : \mathfrak{N}_i \to \mathfrak{N}_{-i}$ for all $\lambda\in\mathbb{C}_+$.
In particular, $\textsf{Sh}(i)=U$.  By virtue of  \eqref{GDR2} with $\lambda=i$, \ $\mathcal{D}(S)\dot{+}\mathfrak{N}_i=\mathcal{D}(S)\dot{+}(I-U)\mathfrak{N}_i$ that is
possible only for the case $U=0$. Hence, $\textsf{Sh}(\lambda)\equiv{0}$  and \eqref{GDR2} implies
that
\begin{equation}\label{ee3}
\mathfrak{N}_\lambda\subset\mathcal{D}(S)\dot{+}\mathfrak{N}_i, \qquad  \forall\lambda\in\mathbb{C}_+.
\end{equation}

Let us assume that there exists $f_i\in\mathfrak{N}_i$ and $\mu\in\mathbb{C}_+$ such
that $f_i=v+f_\mu+f_{\overline{\mu}}$,  where
$v\in\mathcal{D}(S)$ and $f_{\overline{\mu}}\in\mathfrak{N}_{\overline{\mu}}$ is non-zero.
The last equality can be transformed to $\widetilde{f}_i=\widetilde{v}+{f}_{\overline{\mu}}$
 with the use of \eqref{ee3}.  However, the obtained relation is impossible because
 $Im \ (S^*\widetilde{f}_i, \widetilde{f}_i)=\|\widetilde{f}_i\|^2>0$  and, simultaneously,
 $$
 Im \ (S^*\widetilde{f}_i, \widetilde{f}_i)=Im \ (S^*(\widetilde{v}+{f}_{\overline{\mu}}), \widetilde{v}+{f}_{\overline{\mu}})=Im \ (S^*{f}_{\overline{\mu}}, {f}_{\overline{\mu}})=-(Im \ \mu)\|{f}_{\overline{\mu}}\|^2<0.
 $$
Therefore, $f_i=v+f_\mu$ and  $\mathfrak{N}_i\subset\mathcal{D}(S)\dot{+}\mathfrak{N}_\mu$ for all $\mu\in\mathbb{C}$.
The obtained inequality and \eqref{ee3} justify \eqref{assa8}.

Conversely, if \eqref{assa8} holds, then, due to \eqref{GDR2}, $\mathcal{D}(S)\dot{+}(I-\textsf{Sh}(\lambda))\mathfrak{N}_i\subset\mathcal{D}(S)\dot{+}\mathfrak{N}_i$
that is possible only for $\textsf{Sh}(\lambda)\equiv{0}$.
\end{proof}

\begin{remark} The inclusion \eqref{assa8} and its dual counterpart in $\mathbb{C}_-$:
\begin{equation}\label{assa8b}
\mathfrak{N}_\nu\subset\mathcal{D}(S)\dot{+}\mathfrak{N}_{\xi},  \quad  \mbox{for all} \quad  \nu, \xi \in\mathbb{C}_-.
\end{equation}
are equivalent. Indeed, \eqref{assa8} means that the maximal dissipative operators $A_\lambda$ in \eqref{as908} do not depend on the choice of
$\lambda\in\mathbb{C}_+$, i.e., $A_\lambda\equiv{A}_+$. Therefore, theirs adjoint
$A_\lambda^*=A^*_\mu=A_{\nu}=A_\xi=A_+^*$   ($\nu=\overline{\lambda}$, $\xi=\overline{\mu}$)  also do not depend on $\nu,   \xi \in\mathbb{C}_-$.
This fact justifies \eqref{assa8b}.
\end{remark}

\begin{corollary}\label{new68}
Simple Phillips  symmetric operators with the same defect numbers are unitary equivalent.
\end{corollary}
\begin{proof}
Let $S$ be a PSO with defect numbers $<m,m>$.
It follows from the proof of Theorem \ref{assa89} that the Shtraus characteristic function $\textsf{Sh}(\cdot)$ of $S$  coincides with the zero operator.
Therefore,  the characteristic function of $S$ calculated in terms of boundary triplet $(\mathfrak{N}_i, \Gamma_-, \Gamma_+)$ (see \eqref{sas6})
is also zero operator  acting in the auxiliary space with the dimension $m$.  Applying now
theorem \ref{assa44} for the case of simple Phillips symmetric operators with the same defect numbers,
we complete the proof.
\end{proof}

\begin{theorem}\label{assa77}
A  symmetric operator $S$ with equal defect numbers is a Phillips symmetric operator if and only if
its defect subspaces $\mathfrak{N}_\lambda$ and
$\mathfrak{N}_\nu$ are mutually orthogonal for any $\lambda\in\mathbb{C}_+$ and $\nu\in\mathbb{C}_-$.
\end{theorem}
\begin{proof} Let  $S$  be a PSO and let $\lambda, \mu \in\mathbb{C}_+$, $\lambda\not=\mu$.
By virtue of \eqref{assa8}, $f_\lambda=u+f_\mu$, where $f_z\in\mathfrak{N}_z$ and $u\in\mathcal{D}(S)$. Therefore,
\begin{equation}\label{ee4}
0=(S^*-\lambda{I})f_\lambda=(S-\lambda{I})u+(\mu-\lambda)f_\mu.
\end{equation}
The obtained relation means that $\mathfrak{N}_\mu\subset\mathcal{R}(S-\lambda{I})$ and, hence $\mathfrak{N}_\mu\perp\mathfrak{N}_{\overline{\lambda}}$.
To prove the orthogonality of $\mathfrak{N}_\mu$ and $\mathfrak{N}_{\overline{\mu}}$  we use again \eqref{ee4} in order to rewrite $f_\lambda=u+f_\mu$ as follows:
$f_\lambda=(\lambda-\mu)(S-\lambda{I})^{-1}f_\mu+f_\mu$.  Let $f_\mu$ be fixed and $\lambda\to\mu$. Then $f_\lambda\to{f}_\mu$ due to the last formula for $f_\lambda$.
Then,
$$
  (f_\mu, f_{\overline{\mu}})=\lim_{\lambda\to\mu}(f_\lambda, f_{\overline{\mu}})=0, \qquad \forall{f_\mu}\in\mathfrak{N}_\mu, \  f_{\overline{\mu}}\in\mathfrak{N}_{\overline{\mu}}.
$$

Conversely, let $\mathfrak{N}_\lambda\perp\mathfrak{N}_\nu$ for any $\lambda\in\mathbb{C}_+$,  $\nu\in\mathbb{C}_-$.
 Assume now that the relation \eqref{assa8} is not true for some  $\mu\in\mathbb{C}_+$. Then there exists $f_\lambda\in\mathfrak{N}_\lambda$
such that $f_\lambda=u+f_\mu+f_{\overline{\mu}}$, where $f_{\overline{\mu}}\not=0$.  Then
$$
(\lambda-\mu)f_\lambda=(S^*-\mu{I})f_\lambda=(S-\mu{I})u+(\overline{\mu}-\mu)f_{\overline{\mu}}.
$$
Due to our assumption $\mathfrak{N}_\lambda\perp\mathfrak{N}_{\overline{\mu}}$ ($\nu=\overline{\mu}$).   Therefore, for any $\gamma_{\overline{\mu}}\in\mathfrak{N}_{\overline{\mu}}$,
$$
0=(\lambda-\mu)(f_\lambda, \gamma_{\overline{\mu}})=(u, (S^*-\overline{\mu}I)\gamma_{\overline{\mu}})+ (\overline{\mu}-\mu)(f_{\overline{\mu}}, \gamma_{\overline{\mu}})=(\overline{\mu}-\mu)(f_{\overline{\mu}}, \gamma_{\overline{\mu}}).
$$
That is possible only for $f_{\overline{\mu}}=0$.  Therefore, the defect subspaces of $S$ satisfy \eqref{assa8} and $S$
is a Phillips symmetric operator.
 \end{proof}
\begin{remark}\label{assa33}
Theorem \ref{assa77} holds true if the weaker condition of orthogonality $\mathfrak{N}_\lambda\perp\mathfrak{N}_{-i}$
will be used instead of  $\mathfrak{N}_\lambda\perp\mathfrak{N}_{\nu}$.   Indeed,  since the inequalities
\eqref{assa8} and \eqref{ee3} are equivalent (see the proof of Theorem \ref{assa89}),  it suffices
to verify that  $\mathfrak{N}_\lambda\perp\mathfrak{N}_{-i}$ implies  \eqref{ee3}.  The required
implication is obtained by repeating the end part of the proof of Theorem \ref{assa77} with $\mu=i$.
\end{remark}
\begin{corollary}\label{bbb111}
A  symmetric operator $S$ in $\mathfrak{H}$ with equal defect numbers $<m,m>$
is a Phillips symmetric operator if and only if  the Hilbert space $\mathfrak{H}$ can be decomposed into the orthogonal sum
$\mathfrak{H}=\mathfrak{H}_1\oplus\mathfrak{H}_2\oplus\mathfrak{H}_3$  of Hilbert spaces $\mathfrak{H}_j$  leaving $S$ invariant
and such that
\begin{equation}\label{assa10}
S=S_1\oplus{S_2}\oplus{S_3},  \qquad S_j=S\upharpoonright_{\mathfrak{H}_j},
\end{equation}
where  $S_1$ and $S_2$ are simple maximal symmetric operators in
 ${\mathfrak H}_1$ and ${\mathfrak H}_2$ with defect numbers  $<m, 0>$ and $<0, m>$,
 respectively and  $S_3$  is a self-adjoint operator in ${\mathfrak H}_3$.
\end{corollary}

\begin{proof}  If $S$ has the decomposition \eqref{assa10},  then
its defect subspaces $\mathfrak{N}_\lambda$  $(\lambda\in\mathbb{C}_+)$  coincide with
the defect subspaces  $\mathfrak{N}_\lambda(S_1)$ of the operator $S_1$
(since $\mathfrak{N}_\lambda(S_2)=\{0\}$ due to the defect numbers $<0,m>$ of $S_2$) and therefore,
$\mathfrak{N}_\lambda\subset\mathfrak{H}_1$.  Similarly,  the defect numbers  $<m, 0>$  of  $S_1$
mean that $\mathfrak{N}_\nu=\mathfrak{N}_\nu(S_2)\subset\mathfrak{H}_2$ for all $\nu\in\mathbb{C}_-$.
Therefore,  $\mathfrak{N}_\lambda\perp\mathfrak{N}_\nu$ and $S$ is a Phillips symmetric operator
due to Theorem \ref{assa77}.

Conversely, each symmetric operator $S$ with equal defect numbers is reduced by the decomposition
\begin{equation}\label{es121}
\mathfrak{H}=\mathfrak{H}_{\alpha}\oplus\mathfrak{H}_{3}, \qquad \mathfrak{H}_{3}=\bigcap_{\forall\mu\in\mathbb{C}_-\cup\mathbb{C}_+}\mathcal{R}(S-\mu{I}),
\end{equation}
where $\mathfrak{H}_{3}$ is the maximal invariant
subspace for $S$ on which the operator $S_{3}=S\upharpoonright_{\mathfrak{H}_{3}}$ is self-adjoint,  while
the subspace $\mathfrak{H}_{\alpha}$ coincides with the closed linear span of all $\ker(S^*-\mu{I})$ and the restriction
$S_{\alpha}=S\upharpoonright_{\mathfrak{H}_{\alpha}}$ gives rise to a simple symmetric operator in $\mathfrak{H}_{\alpha}$ with defect numbers $<m,m>$ \cite[p.9]{GG}.

Assume now that $S$ is a PSO, then its simple counterpart  $S_\alpha$ is also PSO
and $\mathfrak{N}_\mu=\ker(S^*-\mu{I})=\mathfrak{N}_{\mu}(S_\alpha)=\ker(S^*_{\alpha}-\mu{I})$  for all  $\mu\in\mathbb{C}_-\cup\mathbb{C}_+$.
According to Theorem \ref{assa77},  $\mathfrak{N}_{\lambda}(S_\alpha)\perp\mathfrak{N}_{\nu}(S_\alpha)$
($\lambda\in\mathbb{C}_+, \ \nu\in\mathbb{C}_-$).  Therefore, we can decompose  $\mathfrak{H}_{\alpha}=\mathfrak{H}_{1}\oplus\mathfrak{H}_{2}$,
where $\mathfrak{H}_{1}$  and  $\mathfrak{H}_{2}$ coincide with the closed linear spans of defect subspaces $\{\mathfrak{N}_\mu\}_{\mu\in\mathbb{C}_+}$
and defect subspaces  $\{\mathfrak{N}_\nu\}_{\nu\in\mathbb{C}_-}$,  respectively.

To complete the proof we should verify that $S_\alpha=S_1\oplus{S_2}$, where $S_j=S_\alpha\upharpoonright_{\mathfrak{H}_{j}}$ ($j=1,2$)  are maximal symmetric operators in
$\mathfrak{H}_j$ with defect numbers  $<m, 0>$ and $<0,m>$, respectively. To that end we consider
a simple symmetric operator
\begin{equation}\label{assa7}
\textsc{S}=i\frac{d}{dx}, \qquad \mathcal{D}(\textsc{S})=\{u\in{W^1_2(\mathbb{R}, N)} : u(0)=0\}
\end{equation}
acting in the Hilbert space $L_2(\mathbb{R}, N)$, where $N$ is an auxiliary Hilbert space with the dimension $m$.
It is easy to see that the defect subspaces $\mathfrak{N}_\mu$,  $\mathfrak{N}_{\nu}$  ($\mu\in\mathbb{C}_+$,  $\nu\in\mathbb{C}_-$  )  are formed, respectively, by the functions
\begin{equation}\label{neww45}
\chi_{{\mathbb{R}_-}}(x)e^{-i\mu{x}}n,  \qquad \chi_{\mathbb{R}_+}(x)e^{-i{\nu}{x}}n,
\end{equation}
where $n$ runs the Hilbert space $N$ and $\chi_I(x)$ is the characteristic function of the interval $I$.
Therefore,  the defect numbers of  $\textsc{S}$  is  $<m,m>$  and $\textsc{S}$ is  a Phillips symmetric operator
(since $\mathfrak{N}_\mu$  and  $\mathfrak{N}_{\nu}$ are mutually orthogonal).

By Corollary \ref{new68},  the symmetric operator $S_\alpha$  in $\mathfrak{H}_\alpha$
 is unitary equivalent to the  symmetric operator
$\textsc{S}$  acting in $L_2(\mathbb{R}, N)$.  For this reason, it sufficient to establish
the decomposition $S_\alpha=S_1\oplus{S_2}$  for the case where $S_\alpha=\textsc{S}$ and $\mathfrak{H}_\alpha=L_2(\mathbb{R}, N)$.
Taking \eqref{neww45} into account, we decide that $\mathfrak{H}_1=L_2(\mathbb{R}_-, N)$  and
$\mathfrak{H}_2=L_2(\mathbb{R}_+, N)$.  Moreover, $\textsc{S}=\textsc{S}_1\oplus{\textsc{S}_2}$, where $\textsc{S}_1=i\frac{d}{dx}$, $\mathcal{D}(\textsc{S}_1)=\{u\in{W^1_2(\mathbb{R}_-, N)} : u(0)=0\}$
is a maximal symmetric operator in $L_2(\mathbb{R}_-, N)$ with defect numbers   $<m, 0>$  while,
 $\textsc{S}_2=i\frac{d}{dx}$, $\mathcal{D}(\textsc{S}_2)=\{u\in{W^1_2(\mathbb{R}_+, N)} : u(0)=0\}$
 is maximal symmetric in  $L_2(\mathbb{R}_+, N)$ with defect numbers  $<0,m>$.
\end{proof}
\begin{remark}\label{bbb21}
 It follows from the proof of Corollary \ref{bbb111} that each simple PSO
 $S$ with defect numbers  $<m,m>$ is unitary equivalent to the symmetric operator $\textsc{S}$
defined by \eqref{assa7}.
\end{remark}

\section{Proper extensions of a Phillips symmetric operator}

It follows from \eqref{assa7} and Remark \ref{bbb21} that  the spectrum  of a simple PSO $S$ is
continuous and it coincides with $\mathbb{R}$.  Furthermore, $\ker(S^*-\lambda{I})=\{0\}$  for  $\lambda\in\mathbb{R}$.
Therefore, each proper extension $A$  of a simple PSO has a continuous spectrum on $\mathbb{R}$ without
embedding eigenvalues. The lack of condition of being simple for a PSO means that the spectra of the corresponding proper
extensions coincides with $\mathbb{R}$  but real point spectrum may appear due to a possible self-adjoint part $S_3$ in \eqref{assa10}.

\begin{proposition}\label{assa107}
The spectrum $\sigma(A)$  of a proper extension ${A}$ of a
Phillips symmetric operator $S$ coincides with one of the following sets:
\begin{itemize}
  \item[(i)]  $\sigma({A})=\mathbb{R}$;
  \item[(ii)] $\sigma({A})=\mathbb{C}_-\cup\mathbb{R}$  \ or \ $\sigma({A})=\mathbb{R}\cup\mathbb{C}_+$;
  \item[(iii)] $\sigma({A})=\mathbb{C}$.
\end{itemize}
\end{proposition}
\begin{proof}
Let us suppose that a proper extension ${A}$  has a complex point $\lambda_0\in\rho({A})$.
Without loss of generality we may assume that $\lambda_0\in\mathbb{C}_-$.
Then, the domain of $A$  admits the presentation
$$
\mathcal{D}({A})=\{f=u+u_{\overline{\lambda}_0}+{\Phi}u_{\overline{\lambda}_0}  \ : \  \forall{u}\in\mathcal{D}(S), \ \forall{u_{\overline{\lambda}_0}}\in\mathfrak{N}_{{\overline{\lambda}_0}} \},
$$
where $\Phi : \mathfrak{N}_{{\overline{\lambda}_0}} \to \mathfrak{N}_{{{\lambda}_0}} $  is a bounded operator defined on
$\mathfrak{N}_{{\overline{\lambda}_0}}$.
The obtained expression can be rewritten in terms of  the boundary triplet \eqref{sas6} with $\mu=\overline{\lambda}_0$:
\begin{equation}\label{assa11}
 \mathcal{D}({A}) =\{f\in\mathcal{D}(S^*) \ : \
\mathbf{T}\Gamma_+f=\Gamma_-f \},
\end{equation}
where ${\mathbf{T}}=V\Phi$  is a bounded operator in the auxiliary Hilbert space $\mathfrak{N}_{\mu}$.

By virtue of \cite[Theorem 4.2]{KK} (see also \cite[Theorem 3]{Kochubei}),
$$
\lambda\in\sigma({A}) \iff  0\in\sigma(\Theta(\lambda)-\mathbf{T}),   \quad \lambda\in\mathbb{C}_+,
$$
$$
\lambda\in\sigma({A}) \iff  0\in\sigma(I-\Theta({\lambda})\mathbf{T}),  \quad  \lambda\in\mathbb{C}_-,
$$
where $\Theta(\cdot)$ is the characteristic function of $S$ associated with the boundary triplet $(\mathfrak{N}_{\mu}, \Gamma_-, \Gamma_+)$.
 Since $S$ is PSO, the characteristic function $\Theta(\cdot)$  is a constant on $\mathbb{C}_+$.
Therefore, $\lambda\in\sigma({A})\iff{0}\in\sigma(\Theta-\mathbf{T})$, where $\Theta(\lambda)=\Theta$ for all $\lambda\in\mathbb{C}_+$.
The obtained relation means that either  $\mathbb{C}_+$ belongs to $\sigma({A})$ or $\mathbb{C}_+\subset\rho({A})$.
Further, due to the assumption above, there is a resolvent point $\lambda_0\in\mathbb{C}_-$  of ${A}$.
Therefore $0\in\rho(I-\Theta^*\mathbf{T})$ and $\mathbb{C}_-\subset\rho({A})$.
Summing up, the spectrum $\sigma({A})$
corresponds to the cases $(i)$ or $(ii)$ in dependence of either ${0}\in\rho(\Theta-\mathbf{T})$  or ${0}\in\sigma(\Theta-\mathbf{T})$.
The case $\lambda_0\in\mathbb{C}_+\cap\rho({A})$ is considered in the same manner.
\end{proof}

\begin{theorem}\label{assa102}
Proper extensions of  a Phillips symmetric operator with real spectra are similar to each other.
\end{theorem}
\begin{proof} By virtue of Corollaries \ref{new68}, \ref{bbb111}, it is sufficient to consider proper extensions of the simple PSO $\textsc{S}$ determined by \eqref{assa7}.  In this case,
$$
\textsc{S}^*f=i\frac{df}{dx}, \quad  \mathcal{D}(\textsc{S}^*)=W^{1}_2(\mathbb{R}\setminus\{0\}, N)
$$
and the defect subspaces $\mathfrak{N}_\mu$,  $\mathfrak{N}_{\overline{\mu}}$  ($\mu\in\mathbb{C}_+$) are formed, respectively, by the functions $\chi_{\mathbb{R}_-}(x)e^{-i\mu{x}}n$ and  $\chi_{\mathbb{R}_+}(x)e^{-i\overline{\mu}{x}}n,$
where $n$ runs the auxiliary Hilbert space $N$.

Let us choose the unitary mapping $V : \mathfrak{N}_{\overline{\mu}} \to \mathfrak{N}_\mu$ in the definition
of boundary triplet \eqref{sas6} as $V\chi_{\mathbb{R}_+}(x)e^{-i\overline{\mu}{x}}n=\chi_{\mathbb{R}_-}(x)e^{-i\mu{x}}n$
and consider the unitary mapping  $W$ between $\mathfrak{N}_\mu$  and $N$ as follows:
$$
W\chi_{\mathbb{R}_-}(x)e^{-i\mu{x}}n=\frac{n}{\sqrt{2(Im \ \mu)}}.
$$
Then the modified boundary triplet  $(W\mathfrak{N}_\mu, W\Gamma_-, W\Gamma_+)$ of the boundary triplet \eqref{sas6}
takes the form $(N, \Gamma_-^1, \Gamma_+^1)$,  where
$$
\Gamma_-^1f=f(0+),  \quad  \Gamma_+^1f=f(0-),   \qquad   f\in\mathcal{D}(\textsc{S}^*).
$$

If a proper extension ${A}$  of  $\textsc{S}$ has real spectrum, then its domain of definition
is determined by the formula  \eqref{assa11}, where ${\mathbf{T}}$ is a bounded operator in $\mathfrak{N}\mu$
with bounded inverse. This means that
\begin{equation}\label{assa103}
{A}=\textsc{S}^*\upharpoonright_{\mathcal{D}({A})},  \quad
\mathcal{D}({A})=\{f\in\mathcal{D}(S^*) \ : \
{T}f(0-)=f(0+) \},
\end{equation}
where $T=W{\mathbf{T}}W^{-1}$ is a bounded operator in $N$ with bounded inverse.

Let $F$ be a bounded operator with  bounded inverse in $N$.
Then, the operator
$$
U_Ff=\left\{\begin{array}{c}
Ff(x), \ x>0 \\
f(x),  \ x<0
\end{array}\right., \qquad f\in{L_2(\mathbb{R}, N)}
$$
preserves these properties as an operator acting in $L_2(\mathbb{R}, N)$  and $U_F^{-1}=U_{F^{-1}}$.
Furthermore, $U_F : \mathcal{D}(\textsc{S}^*) \to \mathcal{D}(\textsc{S}^*)$ and $U_F\textsc{S}^*=\textsc{S}^*U_F$.
These relations and \eqref{assa103} lead to the conclusion that
\begin{equation}\label{assa105}
U_FA_{{T}}f=U_F\textsc{S}^*f=\textsc{S}^*U_Ff=A_{F{T}}U_Ff,  \qquad  f\in\mathcal{D}(A_{{T}}),
\end{equation}
where $A_{{T}}$  denotes the proper extension ${A}$ in \eqref{assa103}.

Let ${A}_j$  be proper extensions of   $\textsc{S}$  with  $\sigma({A}_j)=\mathbb{R}$.
Then they are described in \eqref{assa103} by bounded operators  ${T}_j$ with $0\in\rho({T}_j)$ \ (${A}_j\equiv{A}_{{T}_j}$).
Due to \eqref{assa105},
\begin{equation}\label{assa15}
{A}_{{T}_1}=U^{-1}_F{{A}_{{T}_2}}U_F,  \quad \mbox{with} \quad F={T}_2{T}_1^{-1}.
\end{equation}
Therefore, ${A}_j$ are similar to each other.
\end{proof}
\begin{corollary}\label{assa17}
Self-adjoint extensions of a Phillips symmetric operator $S$ are unitary equivalent to each other.
Precisely, there exists a collection of unitary operators $\mathfrak{U}=\{U_\xi\}_{\xi\in{\mathfrak{I}}}$
($\mathfrak{I}$ is the set of indices) with the properties
$$
U_\xi\in\mathfrak{T} \iff U_\xi^*\in\mathfrak{T}, \qquad  U_\xi{S}=SU_\xi, \quad \forall\xi\in\mathfrak{I}
$$
and such that every pair of self-adjoint extensions $A_1$, $A_2$ of $S$ satisfy the relation
\begin{equation}\label{assa16}
U_\xi{A_1}=A_2U_\xi
\end{equation}
for some $\xi\in\mathfrak{I}$.
\end{corollary}
\begin{proof}  All self-adjoint extensions of the symmetric operator
$\textsc{S}$ are uniquely distinguished in \eqref{assa103} by the set of unitary operators $T$ acting in $N$ (see Section \ref{chap2}).
Therefore, the operators $U_T$ defined above are unitary operators in $L_2(\mathbb{R}, N)$ and
\eqref{assa15} can be rewritten as \eqref{assa16} where $A_i=A_{T_i}$ are self-adjoint extensions of
 $\textsc{S}$ and $\xi={T}_2{T}_1^{-1}$ is unitary operator in $L_2(\mathbb{R}, N)$.

It follows from the definition of $U_T$ that the set of $\mathfrak{U}=\{U_\xi\}_{\xi}$,
where $\xi$ runs the set $\mathfrak{I}$ of unitary operators in $N$ satisfies the conditions
of Corollary \ref{assa17}. Therefore, the proof is complete for the simple PSO $\textsc{S}$ defined by \eqref{assa7}.
 This result is extended to an arbitrary simple PSO $S$ with the use of
 Corollary \ref{new68}.

 The required set $\mathfrak{U}=\{U_\xi\}_{\xi\in{\mathfrak{I}}}$  for the general case of a Phillips symmetric operator
 is obtained on the base of previously constructed (for simple operators) set  by the addition of the identity operator $I_3$ acting in the subspace $\mathfrak{H}_3$ (see \eqref{assa10}) corresponding to the self-adjoint part of $S$.
\end{proof}
\begin{remark}
It follows from the construction of $\mathfrak{U}=\{U_\xi\}_{\xi\in{\mathfrak{I}}}$ in Corollary
\ref{assa17} that the symmetric operator $S$ is $\mathfrak{U}$-invariant. However,
there are no $\mathfrak{U}$-invariant self-adjoint extensions of $S$.  Firstly, an example of such kind
 was constructed by Phillips \cite{PH}.
\end{remark}

\begin{corollary}
Each self-adjoint extension of a simple PSO has Lebesgue spectrum on $\mathbb{R}$.
\end{corollary}
\begin{proof} In view of Corollary \ref{new68},  a simple PSO
is unitary equivalent to the symmetric operator $\textsc{S}$ in \eqref{assa7}.
The momentum operator
\begin{equation}\label{assa101b}
\textsc{A}=i\frac{d}{dx},  \qquad \mathcal{D}(\textsc{A})={W^1_2(\mathbb{R}, N)}
\end{equation}
is a  self-adjoint  extension of $\textsc{S}$ in $L_2(\mathbb{R}, N)$ and it has Lebesgue spectrum on $\mathbb{R}$.
Applying now Theorem \ref{assa102} we complete the proof.
\end{proof}

\section{Phillips symmetric operators as the restriction of self-adjoint ones}
\begin{lemma}
 Let $A$ be a self-adjoint operator in a Hilbert space $\mathfrak{H}$.
Each closed densely defined symmetric operator that can be determined via certain restriction of $A$ is
specified by the formula
 \begin{equation}\label{new19}
S_{\mathcal{L}}=A\upharpoonright_{\mathcal{D}(S_{\mathcal{L}})}, \qquad  \mathcal{D}(S_{\mathcal{L}})=\{u\in\mathcal{D}(A) \ :  \ ((A-iI)u, \gamma)=0, \ \forall{\gamma}\in{\mathcal{L}}\},
\end{equation}
where $\mathcal{L}$ is a linear subspace of $\mathfrak{H}$ such that
${\mathcal L}\cap\mathcal{D}(A)=\{0\}$.
\end{lemma}
\begin{proof}
If $S$ is a closed symmetric densely defined restriction of $A$, then $S=S_{\mathcal{L}}$,
where ${\mathcal{L}}=\mathfrak{H}\ominus\mathcal{R}(S-iI)$.  Conversely, let $S_{\mathcal{L}}$
be defined by \eqref{new19}. Obviously, $S_{\mathcal{L}}$ is a closed symmetric operator and,
for any ${u}\in\mathcal{D}(S_{\mathcal{L}})$ and $p\in\mathfrak{H}$,
$$
(u,p)=((S_{\mathcal{L}}-iI)u, (A+iI)^{-1}p)=((A-iI)u, (A+iI)^{-1}p).
$$
The obtained relation means that $S_{\mathcal{L}}$ is nondensely defined $\iff$ $(A+iI)^{-1}p\in{\mathcal{L}} \ \iff \ {\mathcal L}\cap\mathcal{D}(A)\not=\{0\}$. Therefore, the condition
${\mathcal L}\cap\mathcal{D}(A)=\{0\}$ guarantees that $S_{\mathcal{L}}$ is densely defined.
\end{proof}

The symmetric operator $S_{\mathcal{L}}$ in \eqref{new19} turns out to be a PSO
under certain choice of ${\mathcal{L}}$. To specify the required conditions, we consider the
unitary operator
\begin{equation}\label{new534}
U=(A+iI)(A-iI)^{-1},  \qquad (A=i(U+I)(U-I)^{-1})
\end{equation}
which is the Cayley transform of $A$ and
recall that a subspace ${\mathcal{L}}$ is called \emph{a wandering subspace} of $U$
if $U^n{\mathcal{L}}\perp{\mathcal{L}}$ for any $n\in\mathbb{N}$.

\begin{theorem}\label{assa108}
The following statements are equivalent:
\begin{itemize}
  \item[(i)]  the operator $S_{\mathcal{L}}$ defined by \eqref{new19} is a Phillips symmetric operator;
  \item[(ii)] the subspace $\mathcal{L}$ is a wandering subspace of the unitary operator $U$.
\end{itemize}
\end{theorem}
\begin{proof}  $(ii)\Rightarrow(i)$. Let $\mathcal{L}$ be wandering for $U$.  First of all we should check
that  ${\mathcal L}\cap\mathcal{D}(A)=\{0\}$.  Indeed, for all $f\in{\mathcal L}$,  $(U^nf,f)=0$ and, hence
$\int_0^{2\pi}e^{in\lambda}d(E_\lambda{f}, f)=0$, where $E_\lambda$ is the spectral function of $U$.
By the uniqueness theorem for the Fourier-Stieltjes series, the last equality  means
that $(E_\lambda{f}, f)=\frac{\lambda}{2\pi}\|f\|^2$.

It follows from \eqref{new534} that
$$
A=i\int_0^{2\pi}\frac{e^{i\lambda}+1}{e^{i\lambda}-1}dE_\lambda=\int_0^{2\pi}\cot(\lambda/2)dE_\lambda
$$
with the domain $\mathcal{D}(A)=\{f\in\mathfrak{H} \ : \ \int_{0}^{2\pi}\cot^2(\lambda/2)d(E_\lambda{f}, f)<\infty\}$.
In the case of $f\in{\mathcal L}$,
$$
\int_{0}^{2\pi}\cot^2(\lambda/2)d(E_\lambda{f}, f)=\frac{\|f\|^2}{2\pi}\int_{0}^{2\pi}\cot^2(\lambda/2)d\lambda=\infty.
$$
Therefore,  ${\mathcal L}\cap\mathcal{D}(A)=\{0\}$ and the operator $S_{\mathcal{L}}$ is densely defined.

It follows from \eqref{new19} that the defect subspace $\mathfrak{N}_{-i}$ of $S_{\mathcal{L}}$ coincides with
${\mathcal{L}}$. In order to describe other defect subspaces $\mathfrak{N}_{\alpha}$ of $S_{\mathcal{L}}$
we consider the operator
$$
T_{\alpha}=(A+iI)(A-\alpha{I})^{-1}, \qquad \alpha\in\mathbb{C}_-\cup\mathbb{C}_+.
$$
The formula $\mathfrak{N}_{\alpha}=T_{\alpha}{\mathcal{L}}$ is verified directly with the use of \eqref{new19}.

Using \eqref{new534} we get $T_\alpha=2U[(1+i\alpha)U+(1-i\alpha)I]^{-1}.$
In particular, if $\alpha=\lambda\in\mathbb{C}_+$ the obtained expression for $T_\alpha$
can be rewritten as
\begin{equation}\label{ass14}
T_\lambda=\frac{2U}{1-i\lambda}\left[I-tU\right]^{-1}=\frac{2U}{1-i\lambda}\sum_{n=0}^\infty{t^nU^n},  \quad t=\frac{i\lambda+1}{i\lambda-1}
\end{equation}
since $\|tU\|=|t|<1$.

Since $U^n{\mathcal{L}}\perp{\mathcal{L}}$ for all $n\in\mathbb{N}$,
the relation \eqref{ass14} yields that $T_\lambda{\mathcal{L}}\perp{\mathcal{L}}$ for $\lambda\in\mathbb{C}_+$.
Therefore, $\mathfrak{N}_\lambda\perp\mathfrak{N}_{-i}$. Due to Remark \ref{assa33} and Theorem \ref{assa77}, the
operator $S_{\mathcal{L}}$ is PSO. The implication $(ii)\Rightarrow(i)$ is proved.

$(i)\Rightarrow(ii)$.  If $S_{\mathcal{L}}$ is a PSO, then the decomposition \eqref{assa10} and \eqref{new19}
imply that  $\mathcal{L}$ is a subspace of $\mathfrak{H}_1\oplus\mathfrak{H}_2$. Therefore, it is sufficient to assume that
$S_{\mathcal{L}}$ is a simple symmetric operator.

Another important fact is that we can consider \emph{arbitrary} self-adjoint extension of $S_{\mathcal{L}}$  in \eqref{new19}.
Indeed,  let $A_1$ be  a self-adjoint extension of $S_{\mathcal{L}}$  such that  $A_1\not={A}$.  Then, due to Corollary \ref{assa17}, there exists a unitary operator $U_\xi$ such that $U_\xi{S}=SU_\xi$ and $U_\xi{A}={A_1}U_\xi$.
Hence,  the domain $\mathcal{D}(S_{\mathcal{L}})$ can be described
as
$$
\mathcal{D}(S_{\mathcal{L}})=\{v\in\mathcal{D}(A_1) \ :  \ ((A_1-iI)v, g)=0, \ \forall{g}\in{\mathcal{L}_1}=U_\xi\mathcal{L}\}.
$$
Since the Cayley transformations $U$ and $U_1$ of the operators $A$ and $A_1$ are related as
$U_\xi{U}={U_1}U_\xi$, the existence  of a wandering subspace $\mathcal{L}$ for $U$ implies that
${\mathcal{L}_1}=U_\xi\mathcal{L}$ will be a wandering subspace for $U_1$.
Therefore, it does not matter which self-adjoint extension of $S_{\mathcal{L}}$  we will
consider in  \eqref{new19}.

Simple Phillips symmetric operators with the same defect numbers are unitary equivalent  (Corollary \ref{new68}).
For this reason, we can consider a concrete  Phillips symmetric operator in \eqref{new19}.
It is useful to work with PSO $S_{\mathcal{L}}$ defined
in ${\mathfrak H}=l_2(\mathbb{Z}, N)$  ($N$ is an auxiliary Hilbert space) as follows:
\begin{equation}\label{assa4bb}
S_{\mathcal{L}}u=i(\ldots, x_{-3}+x_{-2}, x_{-2}+x_{-1}, \underline{x_{-1}},
x_{1}, x_{1}+x_{2}, \ldots),  \qquad x_j\in{N},
\end{equation}
where element at the zero position is underlined and
$$
u\in\mathcal{D}(S)\iff{u}=(\ldots, x_{-3}-x_{-2},
x_{-2}-x_{-1}, \underline{x_{-1}},-x_{1}, x_{1}-x_{2}, \ldots),
$$
where $\sum_{i\in\mathbb{Z}}\|x_i\|_N^2<\infty$.

The operator $S_{\mathcal{L}}$ defined by \eqref{assa4bb} is a simple PSO
in $l_2(\mathbb{Z}, N)$ and the operator
\begin{equation}\label{e4}
Au=i(\ldots, x_{-3}+x_{-2}, x_{-2}+x_{-1}, \underline{x_{-1}+x_{0}},
x_{0}+x_{1}, x_{1}+x_{2}, \ldots)
\end{equation}
with the domain of definition $u\in\mathcal{D}(A)\iff$
$$
u=(\ldots, x_{-3}-x_{-2}, x_{-2}-x_{-1},
\underline{x_{-1}-x_{0}}, x_{0}-x_{1}, x_{1}-x_{2}, \ldots), \quad \sum_{i\in\mathbb{Z}}\|x_i\|_N^2<\infty
$$
is the self-adjoint extension of $S_{\mathcal{L}}$ \cite{KK, KSV}.

It follows from \eqref{assa4bb} and \eqref{e4} that
$\mathcal{D}(S_{\mathcal{L}})$ consists of those $u\in\mathcal{D}(A)$ for which $x_0=0$.
Direct calculation with use of \eqref{e4} shows that $
(A-iI)u=2i(\ldots, x_{-2},  x_{-1},  \underline{x_{0}},  x_{1},  x_{2}, \ldots).$
Therefore,  the formula \eqref{new19}  gives  $\mathcal{D}(S_{\mathcal{L}})$ if
$$
\mathcal{L}=\{(\ldots, 0,  0,  \underline{x_{0}},  0,  0, \ldots)\in{l_2(\mathbb{Z}, N)} \ : \  \forall{x_0}\in{N}\}.
$$

It is easy to see that the Cayley transform $U$ of $A$ coincides with the bilateral shift
$$
U(\ldots, x_{-2}, x_{-1}, \underline{x_{0}}, x_{1},
x_{2},\ldots)=(\ldots, x_{-3}, x_{-2}, \underline{x_{-1}}, x_{0},
x_{1},\ldots)
$$
in $l_2(\mathbb{Z}, N)$. The subspace $\mathcal{L}$ is a wandering subspace for $U$.
The proof is complete.
\end{proof}
\begin{corollary}\label{bbb27}
The set of symmetric restrictions of a given self-adjoint operator $A$ contains
a  Phillips symmetric operator if and only if there exists  a  reducing subspace $\mathfrak{H}_0$
of $A$ such that the self-adjoint operator $A_0=A\upharpoonright_{\mathcal{D}(A)\cap\mathfrak{H}_0}$
has Lebesgue spectrum on $\mathbb{R}$.

The existence of a simple PSO among symmetric restrictions of $A$ is equivalent to the  fact  that
$A$  has Lebesgue spectrum on $\mathbb{R}$.
\end{corollary}
\begin{proof}
Due to Theorem \ref{assa108}, a PSO can appear only in the case where
there exists a wandering subspace $\mathcal{L}$ of $U$.
Denote $\mathfrak{H}_0=\sum_{n\in\mathbb{Z}}\oplus{U^n}\mathcal{L}.$
The operator $U_0=U\upharpoonright_{\mathfrak{H}_0}$  is a bilateral shift in $\mathfrak{H}_0$ and
its Cayley transform $A_0$ has Lebesgue spectrum on $\mathbb{R}$.

According to Corollary \ref{bbb111},  the existence of a simple PSO means that $\mathfrak{H}_3=\{0\}$
in the decomposition \eqref{assa10}. Therefore,  the corresponding wandering subspace $\mathcal{L}$
(which determines a simple PSO with the help of \eqref{new19}) has the property that
$\mathfrak{H}=\sum_{n\in\mathbb{Z}}\oplus{U^n}\mathcal{L}$. This means that
$A$ has Lebesgue spectrum on $\mathbb{R}$.
\end{proof}

\begin{corollary}
Let  $W_t=e^{iAt}$  be the group of unitary operators generated by a self-adjoint operator $A$.
If there exists  a nonzero $h\in\mathfrak{H}$  such that
$$
(W_th, h)=0, \quad \mbox{for all} \quad  |t|>c,
$$
then the set of symmetric restrictions of $A$ contains a PSO.
\end{corollary}
\begin{proof}  Denote by $\mathfrak{H}_0$  the closure of linear span of $\{{W_t}h\}$  in $\mathfrak{H}$.
The subspace $\mathfrak{H}_0$  reduces $W_t$  and, by virtue of \cite[Lemma 1.2]{Sinaj},  the restriction $W_t\upharpoonright_{\mathfrak{H}_0}$ is a bilateral shift in $\mathfrak{H}_0$. Its generator $A_0$  is
a self-adjoint operator in   $\mathfrak{H}_0$ with the Lebesque spectrum. Applying now Corollary \ref{bbb27}
we complete the proof.
\end{proof}

The concept of Lebesgue spectrum on $\mathbb{R}$  for a self-adjoint operator $A$ can be defined in various (equivalent) ways which
guarantee that the spectral type of $A$ is equivalent to the Lebesgue one and the multiplicity of the spectrum $\sigma(A)$
does not change for any real point.  The last condition is obviously satisfied when $A$ is unitary equivalent to its shifts $A-tI$ for
any $t\in\mathbb{R}$. Development of this `translation-invariance' idea leads to the prominent Weyl commutation relation which
ensures the Lebesgue spectrum property of $A$. Namely, due to the von Neumann theorem \cite[p. 35]{LF},
 a self-adjoint operator $A$ in $\mathfrak{H}$  has
the Lebesgue spectrum on $\mathbb{R}$ if and only if there exists a strongly continuous group of unitary operators $V_t$ such that
\begin{equation}\label{bbb13}
V_tAV_{-t}=A-tI, \qquad \forall{t}\in\mathbb{R}.
\end{equation}

It should be mention that any simple PSO is also a solution of the Weyl commutation relation \eqref{bbb13}.
Indeed,  any operator $A$ which is the solution of \eqref{bbb13} is determined up to unitary equivalence.
Therefore, it is sufficient  to consider the simple PSO  $\textsc{S}$ defined by \eqref{assa7} in $L_2(\mathbb{R}, N)$ and
 to verify that $A=\textsc{S}$  is a solution of \eqref{bbb13} with $V_t$ acting as the multiplication on $e^{-itx}$.

\section{Examples of PSO}

{\bf I.}  Let  $S_{\mathcal{L}}$  be a PSO that is determined by \eqref{new19} as the restriction of
a self-adjoint operator $A$.  Consider a unitary operator $W$  that commutes with $A$.
It is easy to see  that  $S'=WS_{\mathcal{L}}W^{-1}$  is also PSO and $S'$  is determined  by
\eqref{new19} with the new wandering subspace $\mathcal{L}'=W\mathcal{L}$,  i.e.,
$S'=S_{W\mathcal{L}}$.  This simple observation gives rise to infinitely many PSO
 which are symmetric restrictions of a given self-adjoint operator $A$.

If  $A$ has Lebesgue spectrum on $\mathbb{R}$ and its multiplicity coincides with $\dim\mathcal{L}$,
then the obtained  PSO is simple. Furthermore,
 the space $\mathfrak{H}$ is presented as  $\mathfrak{H}=\sum_{n\in\mathbb{Z}}\oplus{U^n{\mathcal{L}}}$.
The last decomposition allows one to determine a unitary mapping of
 $\mathfrak{H}$  onto   $L_2(\mathbb{R}, {\mathcal L})$ in such a way that
$A$ corresponds to the multiplication by independent variable:  $Af(\delta)=\delta{f(\delta)}$;
the Cayley transform of $A$ acts as the multiplication operator:  $Uf(\delta)=\frac{\delta+iI}{\delta-iI}f(\delta)$;
and the wandering subspace  $\mathcal L$  (in $\mathfrak{H}$)  is mapped onto the wandering
subspace $\frac{1}{\delta+i}{\mathcal L}$  in $L_2(\mathbb{R}, {\mathcal L})$ \cite[Chapter 2]{LF}.

If $W$ is a unitary operator in $L_2(\mathbb{R}, {\mathcal L})$  that commutes with $A$, then $W$
can be realized as a multiplicative operator-valued function $w(\delta)$ on ${\mathcal L}$ into ${\mathcal L}$
which is unitary for almost all $\delta$ (see, e.g.,  \cite[Corollary 4.2, p.53]{LF}) :
$$
Wf=w(\delta)f(\delta), \qquad  f\in{L_2(\mathbb{R}, {\mathcal L})}.
$$
This means that the subspaces
${\mathcal L}_{w}\equiv{W}\frac{1}{\delta+i}{\mathcal L}=\frac{w(\delta)}{\delta+i}{\mathcal L}$ are wandering in ${L_2(\mathbb{R}, {\mathcal L})}$
and they determines infinitely many simple PSO $S_{w}$, which, due to \eqref{new19}, are  restrictions of the operator $A$ of multiplication by
$\delta$ onto linear manifolds
$$
\mathcal{D}(S_{w})=\{u\in{\mathcal{D}(A)} :  \int_{\mathbb{R}} (u(\delta), w(\delta)v)_{\mathcal{L}}d\delta=0, \ \forall{v}\in{\mathcal{L}} \}.
$$

 This result can be reformulated for the restrictions of self-adjoint momentum operator  $\textsc{A}$
 (see \eqref{assa101b} where $\mathcal{L}=N$)
with the use of Fourier transformation $(Ff)(x)=\frac{1}{\sqrt{2\pi}}\int_{\mathbb{R}}e^{-i\delta{x}}f(\delta)d\delta$ that
relates $Af(\delta)=\delta{f(\delta)}$ and $\textsc{A}f=i\frac{df}{dx}$. Taking into account that
$\textsc{A}F=FA$   we decide that the subspaces $F{\mathcal L}_{w}$ are wandering for the Cayley
transform of $\textsc{A}$ in $L_2(\mathbb{R}, \mathcal{L})$.
Simple PSO $\textsc{S}_{w}$ are the restrictions of  $\textsc{A}$ onto those functions $f\in{W^1_2(\mathbb{R}, \mathcal{L})}$  that satisfy the relation (c.f. \eqref{new19}):
$((\textsc{A}-iI)u, \gamma)=0, \  \forall{\gamma}\in{F{\mathcal L}_{w}}.$  It is clear that
$\textsc{S}_{w}=FS_{w}F^{-1}$.

Let us set $w(\delta)\equiv{1}$. Then the wandering subspace $F{\mathcal L}_w$ coincides with the subspace
$\chi_{\mathbb{R}_+}(x)e^{-x}{\mathcal L}$ and the formula \eqref{new19} leads, to the simple PSO $\textsc{S}$
defined by \eqref{assa7}. The operator $\textsc{S}$ is the result of one-point perturbation of the momentum operator
$\textsc{A}$ supported at point $x=0$.  The symmetric operator
\begin{equation}\label{neww12}
\textsc{S}_{w}=i\frac{d}{dx}, \qquad \mathcal{D}(\textsc{S}_w)=\{u\in{W^1_2(\mathbb{R}, \mathcal{L})} : u(y)=0\}
\end{equation}
corresponding to one-point interaction supported at  real point $x=y$ is deduced from the formulas
above with  $w(\delta)=e^{i\delta{y}}$.

Assume now that $w(\delta)=\frac{\delta+iI}{\delta-iI}$.  Then
$F{\mathcal L}_w=F\left(\frac{1}{\delta-iI}{\mathcal L}\right)=\chi_{\mathbb{R}_-}(x)e^{x}{\mathcal L}$ and
the formula \eqref{new19} gives rise to the simple PSO
\begin{equation}\label{new1w}
\textsc{S}_{w}=i\frac{d}{dx}, \qquad \mathcal{D}(\textsc{S}_w)=\{u\in{W^1_2(\mathbb{R}, {\mathcal L})} :  u(0)=2\int_{-\infty}^{0}u(x)e^xdx \},
\end{equation}
which is an example of nonlocal point interaction of the momentum operator $\textsc{A}$.

{\bf II.} Let $Df=\sqrt{2}f(2x)$ and $Tf=f(x-1)$ be the dilation and the translation operators in $L_2(\mathbb{R})$.
Denote  $A=i(D+I)(D-I)^{-1}$.  The operator $A$ is self-adjoint in $L_2(\mathbb{R})$ and
it has Lebesgue spectrum on $\mathbb{R}$ (since $D$ is a bilateral shift).
\begin{proposition}\label{neww35}
Let  $S$ be a simple PSO that is a restriction of
the self-adjoint operator $A$ and let  $\psi\in\mathfrak{N}_{-i}=\ker(S^*+iI)$ be a function such that $\{T^{k}\psi\}_{k\in\mathbb{Z}}$
is an orthonormal basis of $\mathfrak{N}_{-i}$. Then $\psi$ is a wavelet.
\end{proposition}
\begin{proof}
If $S$ is a restriction of $A$, then $S$ is determined by \eqref{new19}, i.e.,
$S=S_{\mathcal{L}}$, where $\mathcal{L}=\mathfrak{N}_{-i}$.  Since
$S$ is assumed to be PSO, Theorem \ref{assa108}  implies that $\mathcal{L}$ is a wandering subspace
for the dilation operator $D$. Moreover, the simplicity of $S$ means that
$L_2(\mathbb{R})=\sum_{n\in\mathbb{Z}}\oplus{D^j{\mathcal{L}}}$.
Therefore,  $\{D^jT^k\psi\}_{j, k\in\mathbb{Z}}$ is an orthonormal basis of $L_2(\mathbb{R})$,
i.e., $\psi$ is a wavelet \cite{CHR}.
\end{proof}

\section{Nonlocal point interactions}
The above results show that one-point interaction of the momentum operator:
 $\textsc{A}+\alpha\delta(x-y)$
leads to self-adjoint operators  which are unitary equivalent to each other and have Lebesgue spectrum on $\mathbb{R}$.
This means that non-trivial spectral properties of
self-adjoint operators associated with the momentum operator should be obtained
with the help of more complicated  perturbations. In the present section we consider
special classes of general nonlocal one point interactions \cite{Nizhnik} which can be characterized as
one point interaction defined by the nonlocal potential $\gamma(x)\in{L_2(\mathbb{R})}$.

{\bf I.}  Let us consider  the maximal operator $S_{max}$  which  is  determined on $W^{1}_2(\mathbb{R}\setminus\{0\})$  by the differential expression
$$
S_{max}f=i\frac{df}{dx}+\gamma(x)f_r  \quad  (x\not=0),  \quad  f_r=\frac{1}{2}(f(0+)+f(0-)),
$$
where the non-local potential $\gamma(x)$ belongs to $L_2(\mathbb{R})$.
Direct calculation shows that, for all $f, g\in\mathcal{D}(S_{max})=W^{1}_2(\mathbb{R}\setminus\{0\})$,
$$
(S_{max}f, g)-(f, S_{max}g)=i[\Gamma_{+}f\overline{\Gamma_{+}g}
-\Gamma_{-}f \overline{\Gamma_{-}g}],
$$
where  ${\Gamma}_{\pm}$ are determined by
\begin{equation}\label{new59}
\Gamma_{+}f=f(0-)+\frac{i}{2}(f, \gamma),  \qquad \Gamma_{-}f=f(0+)-\frac{i}{2}(f, \gamma).
\end{equation}
\begin{lemma}\label{assa67}
The operator
$$
 S_{min}=S_{max}\upharpoonright_{\mathcal{D}(S_{min})},  \qquad  \mathcal{D}(S_{min})=\ker\Gamma_-\cap\ker\Gamma_+
 $$
is a closed densely defined symmetric operator in $L_2(\mathbb{R})$ and such that
$S_{min}^*=S_{max}$. A triplet $(\mathbb{C}, \Gamma_{-},  \Gamma_{+})$ , where
the linear mappings ${\Gamma}_{\pm} : W^{1}_2(\mathbb{R}\setminus\{0\})  \to \mathbb{C}$  are determined by
\eqref{new59}  is a boundary triplet of $S_{max}$.
 \end{lemma}
\begin{proof}
To complete the proof, by virtue of \cite[Corollary 2.5]{Behrndt} and Remark \ref{rem22},  it is sufficient to verify that:
(i)  there is a unimodular $c$ such that
the operator $A=S_{max}\upharpoonright_{\ker(c\Gamma_+-\Gamma_-)}$
is self-adjoint in $L_2(\mathbb{R})$;  (ii) the map $(\Gamma_-, \Gamma_+) : \mathcal{D}(S_{max}) \to \mathbb{C}^2$ is surjective.

The condition (i) is satisfied if we choose $c=-1$. In this case $A=i\frac{d}{dx}$ with the domain
$\mathcal{D}(A)=\{f{\in}W^{1}_2(\mathbb{R}\setminus\{0\}) : f(0-)=-f(0+)\}$ is a self-adjoint operator.

Let $h=(h_1, h_2)$ be an arbitrary element of $\mathbb{C}^2$. There exists $f\in{W^{1}_2(\mathbb{R}\setminus\{0\})}$
such that $f(0-)=h_1$ and $f(0+)=h_2$. Let us fix $u\in{W^{1}_2(\mathbb{R}\setminus\{0\})}$  such that $u(0-)=u(0+)=0$ and
$(u, \gamma)\not=0$.
Using now \eqref{new59}  we decide that the vector $\widetilde{f}=f-\frac{(f, \gamma)}{(u, \gamma)}u$
solves the equation $(\Gamma_-, \Gamma_+)\widetilde{f}=(h_1, h_2)$, that justifies (ii).
\end{proof}

The boundary triplet  $(\mathbb{C}, \Gamma_{-},  \Gamma_{+})$  constructed  in  Lemma \ref{assa67}
allows us to determine self-adjoint operators
\begin{equation}\label{neww67}
{A}_\theta{f}=i\frac{df}{dx}+\gamma(x)f_r,   \quad  f\in\mathcal{D}(A_\theta)\subset{W^{1}_2(\mathbb{R}\setminus\{0\})},  \quad \theta\in[0, 2\pi)
\end{equation}
whose domains $\mathcal{D}(A_\theta)$ consist of all functions $f\in{W^{1}_2(\mathbb{R}\setminus\{0\})}$  that satisfy the
nonlocal boundary-value condition
$$
 e^{i\theta}[f(0-)+\frac{i}{2}(f, \gamma)]=f(0+)-\frac{i}{2}(f, \gamma).
$$
These operators are mathematical models of one point interaction defined by the nonlocal potential $\gamma(x)$.

Each operator $A_\theta$  is a self-adjoint extension of the symmetric operator
$S_{min}=S^*_{max}= i\frac{d}{dx}$   with
domain of definition
\begin{equation}\label{assa24}
\mathcal{D}(S_{min})=\left\{f\in{W}^{1}_2(\mathbb{R}\setminus\{0\}) \  : \  \begin{array}{c}
f(0-)+\frac{i}{2}(f, \gamma)=0  \\
f(0+)-\frac{i}{2}(f, \gamma)=0
\end{array} \right\}.
\end{equation}

The symmetric operator $S_{min}$ has defect numbers $<1,1>$ and its defect subspaces $\mathfrak{N}_\lambda$,
$\mathfrak{N}_\nu$ ($\lambda\in\mathbb{C}_+$,  $\nu\in\mathbb{C}_-$) coincide with the linear span of the vectors
$$
f_\lambda(x)=g^\lambda(x)-2[1+g^\lambda(0)]G^+_\lambda(x)  \quad  \mbox{and} \quad
f_\nu(x)=g^\nu(x)-2[1+g^\nu(0)]G^-_\nu(x),
$$
respectively. Here
\begin{equation}\label{neww73}
g^z=(\textsc{A}-z{I})^{-1}\gamma=\left\{\begin{array}{l}
ie^{-izx}\int_x^\infty{e^{iz\tau}}\gamma(\tau)d\tau, \quad  z\in\mathbb{C_+}   \smallskip \\
-ie^{-izx}\int^x_{-\infty}{e^{iz\tau}}\gamma(\tau)d\tau, \quad  z\in\mathbb{C_-}
\end{array}\right.
\end{equation}
and
$$
G^+_\lambda(x)=\left\{\begin{array}{ll}
0, & x>0 \\
e^{-i\lambda{x}}  & x<0,
\end{array}\right. \qquad  G^-_\nu(x)=\left\{\begin{array}{ll}
e^{-i\nu{x}}, & x>0 \\
0  & x<0,
\end{array}\right.
$$

By virtue of \eqref{new1w} and \eqref{new59},  the characteristic function $\Theta(\cdot)$ has the form
\begin{equation}\label{neww13}
\Theta(\lambda)=\frac{\Gamma_-{f_\lambda}}{\Gamma_+{f_\lambda}}=\frac{f_\lambda(0+)-\frac{i}{2}(f_\lambda, \gamma)}{f_\lambda(0-)+\frac{i}{2}(f_\lambda, \gamma)}=-I+\frac{2}{2+g^\lambda(0)-\frac{i}{2}(f_\lambda, \gamma)}.
\end{equation}

Let us consider a particular case assuming that  $\gamma=\alpha\chi_{\mathbb{R_+}}(x)e^{-x}$, $\alpha\in\mathbb{C}$.
Then
$$
g^\lambda(x)=\frac{i\alpha}{1-i\lambda}\left\{\begin{array}{l}
e^{-x}, \quad x>0 \\
e^{-i\lambda{x}}, \quad x<0
\end{array}\right.
$$
and
$$
g^\lambda(0)-\frac{i}{2}(f_\lambda, \gamma)=\frac{i\alpha}{1-i\lambda}(1-i\overline{\alpha}/4)
$$
Therefore, the characteristic function  \eqref{neww13} turns out to be a constant on $\mathbb{C}_+$ when $\beta=4i$.
In this case,  the symmetric operator $S_{min}$ in \eqref{assa24} is PSO and all its self-adjoint extensions \eqref{neww67}
have Lebesgue spectrum on $\mathbb{R}$.

{\bf II.}  Let  the maximal operator $S_{max}$  be determined by the differential expression
$$
S_{max}f=i\frac{df}{dx}+\gamma(x)f_s  \quad  (x\not=0),  \quad  f_s=f(0+)-f(0-),
$$
where the non-local potential $\gamma(x)$ belongs to $L_2(\mathbb{R})$.
Similarly to the previous case, the Green formula can be established
$$
(S_{max}f, g)-(f, S_{max}g)=i[\Gamma_{+}f\overline{\Gamma_{+}g}
-\Gamma_{-}f \overline{\Gamma_{-}g}],  \quad f, g\in\mathcal{D}(S_{max})=W^{1}_2(\mathbb{R}\setminus\{0\})
$$
where  $\Gamma_{+}f=f(0-)-i(f, \gamma)$  and $\Gamma_{-}f=f(0+)-i(f, \gamma)$.
The same arguments as in the proof of Lemma \ref{assa67} leads to the conclusion that $(\mathbb{C}, \Gamma_{-},  \Gamma_{+})$
is a boundary triplet of $S_{max}$ and the corresponding symmetric operator
 $S_{min}=S_{max}\upharpoonright_{\mathcal{D}(S_{min})}$,   $ \mathcal{D}(S_{min})=\ker\Gamma_-\cap\ker\Gamma_+$
has the form
\begin{equation}\label{neww69}
S_{min}= i\frac{d}{dx}, \quad
\mathcal{D}(S_{min})=\{f\in{W}^{1}_2(\mathbb{R}) \  : \  f(0)=i(f, \gamma)\}.
\end{equation}

Each self-adjoint extension  ${A}_\alpha$ of ${S}_{min}$  is determined by the formula
$$
{A}_\alpha{f}=i\frac{df}{dx}+\gamma(x)f_s,
$$
where  $
\mathcal{D}({A}_\theta)=\{f\in\mathcal{D}({S}_{max}) \ : \ e^{i\theta}[f(0-)-i(f, \gamma)]=f(0+)-i(f, \gamma) \}.$

The defect subspaces $\mathfrak{N}_\lambda$,
$\mathfrak{N}_\nu$ ($\lambda\in\mathbb{C}_+$,  $\nu\in\mathbb{C}_-$) of $S_{min}$  coincide with the linear span of vectors
$$
f_\lambda(x)=g^\lambda(x)+G^+_\lambda(x)  \quad  \mbox{and} \quad
f_\nu(x)=g^\nu(x)-G^-_\nu(x),
$$
respectively.
Let us fix  $\gamma=\alpha\chi_{\mathbb{R_-}}(x)e^{x}$ and
specify for which $\alpha\in\mathbb{C}$ the corresponding symmetric operator $S_{min}$ will be PSO.
It follows from \eqref{neww73} that
$$
g^\lambda(x)=\frac{i\alpha\chi_{\mathbb{R}_-}(x)}{1+i\lambda}(e^{-i\lambda{x}}-e^{x}),  \quad  g^\nu(x)=-\frac{i\alpha}{1+i\nu}\left\{\begin{array}{l}
e^{-i{\nu}x}, \quad x>0 \\
e^{x}, \quad x<0.
\end{array}\right.
$$
The obtained expressions  allows  one to calculate
$$
(f_\lambda, f_\nu)=\frac{\overline{\alpha}}{2(1-i\lambda)(1-i\overline{\nu})}(2i-\alpha),   \qquad \lambda\in\mathbb{C}_+, \quad  \nu\in\mathbb{C}_-.
$$
The obtained expression and Theorem \ref{assa77} mean that the symmetric operator $S_{min}$ defined in \eqref{neww69} is PSO if and only if $\alpha=2i$.  In this case,
the PSO $S_{min}$ coincides with the operator $\textsc{S}_{w}$  determined by \eqref{new1w}.

\bibliographystyle{amsplain}

\begin{thebibliography}{99}
\bibitem{AL} S.~Albeverio, F.~Gesztesy, R.~H{\o}egh-Krohn, and H.~Holden, \emph{Solvable Models in Quantum Mechanics}, Springer-Verlag, Berlin/New York, 1988;  $2^{\mathrm{nd}}$ ed. (with an appendix by P.~Exner), AMS Chelsea Publishing, Providence, RI, 2005.

\bibitem{AK_AK} S.~ Albeverio and S.~Kuzhel,  \emph{$\mathcal{PT}$-symmetric operators in quantum mechanics: Krein spaces methods}, Non-selfadjoint operators in quantum physics, 293--343, Wiley, Hoboken, NJ, 2015.

\bibitem{Nizhnik} S. Albeverio and L.P. Nizhnik,  \emph{Schr\"{o}dinger operators with nonlocal potentials} Meth. Funct. Anal. Topology,  {\bf 19} (2013),  199--210.

\bibitem{Arlin} Yu. M. Arlinskii,  V. A. Derkach and  E. R. Tsekanovskii, \emph{On unitary equivalent quasi-Hermitian extensions of Hermitian operators} 
Mat.  Fiz., {\bf 29}  (1981),  72--77  ( in Russian).

\bibitem{AZ} T.\ Ya.~Azizov and I.S.~Iokhvidov,  \emph{Linear Operators in Spaces with an Indefinite Metric},
John Wiley \& Sons, Chichester, 1989.

\bibitem{Behrndt}  J. Behrndt and M. Langer,  \textit{On the adjoint of a symmetric operator},  J. Lond. Math. Soc. \textbf{82} (2010), 563–-580.

\bibitem{CHR} O. Christensen, \textit{Functions, Spaces, and Expansions}, Birkh\"{a}user, Basel, 2010.

\bibitem{exner} P. Exner, \emph{Momentum operators on graphs}, Spectral Analysis, Differential Equations and Mathematical Physics: A Festschrift in Honor of Fritz Gesztesy’s 60th Birthday
{\bf 87} (2012), 105--118.

\bibitem{Gor} M. L. Gorbachuk and V. I. Gorbachuk, {\it Boundary-Value Problems for Operator-Differential Equations,} Kluwer, Dordrecht, 1991.

\bibitem{GG} M.\ L.\ Gorbachuk, V.\ I.\ Gorbachuk,  \emph{M.G.\ Krein's Lectures on Entire Operators}, Operator Theory Advances and Applications {\bf 97}, Birkh\"{a}user, Basel, 1997.

\bibitem{HK}  S. Hassi and S. Kuzhel, \emph{On $J$-self-adjoint operators with stable $\mathcal{C}$-symmetries}, Proc. Royal Soc. Edinburgh {\bf 143A} (2013), 141-167.

\bibitem{pedersen1} P. Jorgensen,  S. Pedersen, and F. Tian,   \emph{Momentum operators in two intervals: Spectra and phase transition},
Complex Anal. Operator Theory {\bf 7} (2013), 1735--1773.

\bibitem{KO} A. N. Kochubei, \emph{About symmetric operators commuting with a family of unitary operators}, Funk. Anal. Prilozh. {\bf 13}(1979) 77--78.

\bibitem{Kochubei} A.\ N.\ Kochubei, \emph{On extensions and characteristic functions of symmetric operators}, Izv.\ Akad.\ Nauk.\ Arm.\ SSR {\bf 15} (1980), 219--232. (In Russian); English translation: Soviet J.\ Contemporary Math.\ Anal.\ {\bf 15} (1980).

\bibitem{KK}  A. Kuzhel and S. Kuzhel, {\it Regular Extensions of Hermitian Operators,} VSP, Utrecht, 1998.

\bibitem{KSV} S. Kuzhel, O. Shapovalova, and L. Vavrykovych, \emph{On $J$-self-adjoint extensions of the Phillips symmetric operator}, Meth. Funct. Anal. Topology,
{\bf 16} (2010), 333--348.

\bibitem{LF} P. D. Lax and R. F.g Phillips, \emph{Scattering Theory. Revised Edition}, Academic Press, Inc. San Diego, 1989.

\bibitem{Nizhnik2} L.P. Nizhnik, \emph{Inverse spectral nonlocal problem for the first order ordinary differential equation}, Tamkang J. Math. {\bf 42} (2011) 385--394.

\bibitem{Oliveira} C. R. de Oliveira, \emph{Intermediate Spectral Theory and Quantum Dynamics,}  Birkhäuser, Basel, 2009.

\bibitem{pedersen2} S. Pedersen, J. D. Phillips, F. Tian, C. E. Watson,  \emph{On the spectra of momentum operators}, Complex Anal. Oper. Theory {\bf 9} (2015), 1557--1587.

\bibitem{PH} R. S. Phillips, \emph{The extension of dual subspaces invariant under an algebra}, in: Proceedings of the International Symposium on Linear Spaces (Jerusalem, 1960), pp. 366-398, Jerusalem Academic Press, 1961.

\bibitem{Schm} K. Schm\"{u}dgen,  \emph{Unbounded Self-adjoint Operators on Hilbert space}, Springer, Berlin, 2012.

\bibitem{SH} A. V. Shtraus, \emph{On extensions and characteristic functions of symmetric operators}, Izv. Akad. Nauk SSSR. Ser. Mat. {\bf 32} (1968), 186--207. (Russian)

\bibitem{Sinaj}  Ja. G. Sinai, \emph{Dynamical systems with countable Lebesgue spectrum. I.},  Izv. Akad. Nauk SSSR Ser. Mat. {\bf 25} (1961) 899-–924. (Russian)
\end{thebibliography}

\end{document}